\documentclass[11pt]{amsart} 
\usepackage{verbatim, latexsym, amssymb, amsmath,color,}
\usepackage[cal=boondox]{mathalfa}
\usepackage{epsfig}

\def\R{\mathbb R}
\def\H{\mathbb H}

\def\N{\mathbb N}
\def\Z{\mathbb Z}
\def\C{\mathbb C}

\def\vol{\mathrm{vol}}

\def\area{\mathrm{area}}

\newcommand{\floor}[1]{\left\lfloor #1 \right\rfloor}

\newtheorem{thm}{Theorem}[section]
\newtheorem{lemm}[thm]{Lemma}
\newtheorem{cor}[thm]{Corollary}
\newtheorem{prop}[thm]{Proposition}
\theoremstyle{remark}

\theoremstyle{definition}
\newtheorem{defi}[thm]{Definition}

\title{Counting minimal surfaces in negatively curved $3$-manifolds}
\author{Danny Calegari, Fernando C. Marques, and Andr\'e Neves}
\address{University of Chicago \\ Department of Mathematics \\ Chicago IL 60637\\ USA}
\email{aneves@math.uchicago.edu}
\address{Princeton University \\ Fine Hall \\ Princeton NJ 08544 \\ USA}
\email{coda@math.princeton.edu}
\address{University of Chicago \\ Department of Mathematics \\ Chicago IL 60637\\ USA}
\email{aneves@math.uchicago.edu}
\thanks{The second author is partly supported by NSF-DMS-1811840. The third author is partly supported by NSF  DMS-1710846 and a Simons Investigator Grant.}

\begin{document}
\maketitle

\begin{abstract}
We introduced an asymptotic quantity that counts area-minimizing surfaces in negatively curved closed $3$-manifolds and show that quantity to only be minimized, among all metrics of sectional curvature $\leq -1$, by the hyperbolic metric.

 \end{abstract}
 
\section{Introduction}
A classical and beautiful result in geometry says that if $(M,h_0)$ is a closed locally symmetric Riemannian manifold with {strictly negative curvature} and $h$ {another negatively curved} Riemannian metric on $M$ with the {same volume as $h_0$}, then the quantity
$$\delta(h):=\lim_{L\to\infty}\frac{\ln\#\{{\rm length}_h(\gamma)\leq L:\,\gamma\text{ closed geodesic in }(M,h)\}}{L}$$
satisfies $\delta(h)\geq \delta(h_0)$ and equality  implies that $h$ is isometric to $h_0$. 

This follows from combining a theorem of Margulis \cite{margulis} which identified the right hand side in the inequality above as the topological entropy for negatively curved metrics, a theorem of Manning \cite{manning} which says that the volume entropy and topological entropy coincide for negatively curved metrics, and a theorem of Besson--Courtois--Gallot \cite{BCG} which says that $g_0$ minimizes the volume entropy among all metric with the same volume.

Closed geodesics are a particular case of minimal surfaces and in the last years great progress has been made regarding existence of minimal hypersurfaces. For instance, for closed Riemannian manifold $M$ of dimension between $3$ and $7$, Irie and the last two authors \cite{irie-marques-neves} showed that, for generic metrics, the set of all closed embedded minimal hypersurfaces is dense in $M$; jointly with Song \cite{marques-neves-song} the last two authors showed that, for generic metrics, there is a sequence of closed embedded minimal hypersurfaces that becomes equidistributed; Song \cite{song} showed that for every Riemannian metric on $M$, there are always infinitely many distinct closed embedded minimal hypersurfaces; Zhou \cite{zhou} solved the Multiplicity One Conjecture made by the last two authors, which when combined with \cite{marques-neves-lower} implies that, for generic metrics, there is a closed embedded minimal hypersurface of Morse index $p$ for every $p\in\N$. 

The purpose of this paper is to study minimal surfaces in  a  closed orientable  $3$-manifold in the spirit of the entropy functional mentioned  at the beginning of the introduction.

 Before we state the main theorem we need to introduce some concepts. Throughout this paper, $M$ will denote a closed orientable $3$-manifold that admits an hyperbolic metric. A closed immersed genus $g$ surface $\Sigma\subset M$ is {\em essential} if the immersion $\iota:\Sigma \rightarrow M$ injects  $\pi_1(\Sigma)$ into $\pi_1(M)$. In this case, the group $G=\iota_{*}(\pi_1(\Sigma))$ is called a {\em surface subgroup of genus $g$} and surface subgroups of immersions homotopic to $\iota$ are in one to one correspondence with  conjugates of $G$ by an element of $\pi_1(M)$.

Let $S(M,g)$ denote the set of surfaces subgroups of genus at most $g$ of $\pi_1(M)$ modulo the equivalence relation of conjugacy. We abuse notation and see an element $\Pi\in S(M,g)$ as being either all subgroups of $\pi_1(M)$ that are conjugate to a fixed surface group of genus at most $g$ or the set of all essential immersions of surfaces  $\iota:\Sigma \rightarrow M$ for which $\iota_*(\pi_1(\Sigma))\in \Pi$. Khan and Markovic \cite{kahn-markovic2,kahn-markovic}  showed that  surface subgroups exist for all large genus and estimated the cardinality of $S(M,g)$. 

Consider a Riemannian metric $h$ on $M$ and denote the hyperbolic metric by $\bar h$. Given $\Pi\in S(M,g)$ we define
$$\text{area}_h(\Pi)=\inf\{\text{area}_h(\Sigma): \Sigma\in \Pi\},$$
where $\text{area}_h(\Sigma)$ denotes the area computed with respect to the metric $\iota^*h$. 

Given $\varepsilon\geq 0$ we define $S(M,g,\varepsilon)$ to be the conjugacy classes in $S(M,g)$ whose limit set is a $(1+\varepsilon)$-circle (see Definition  \ref{limit.set.definition}) and set
$$S_{\varepsilon}(M)=\cup_{g\in\N}S(M,g,\varepsilon).$$
We are interested in the following geometric quantity
\begin{equation}\label{asymp.area}
E(h)=\lim_{\varepsilon\to 0}\liminf_{L\to\infty}\frac{\ln \#\{\text{area}_h(\Pi)\leq 4\pi(L-1):\Pi\in S_{\varepsilon}(M)\}}{L\ln L}.
\end{equation}
Note that if $\varepsilon<\varepsilon'$, then $S_{\varepsilon}(M)\subset S_{\varepsilon'}(M)$ a{nd so the limit in the $\varepsilon$-variable is well defined.  
In this paper we show
\begin{thm}\label{main.thm} Given a Riemannian metric $h$ on $M$  {with volume entropy denoted by $E_{vol}(h)$ we have
$E(h)\leq 2E_{vol}(h)^2.$}

If the sectional curvature of $h$ is less or equal than $-1$ then
 $$E(h)\geq E(\bar h)=2 $$
 with equality if and only if $h$ is the hyperbolic metric.
\end{thm}
As far as the authors know, this is the first result giving asymptotic rigidity  for the areas of minimal surfaces.

One obvious challenge is that  the results in \cite{BCG,manning,margulis} rely on the dynamical properties of the geodesic flow, which have no analogue for minimal surfaces. For this reason we restricted our asymptotic counting invariant to the  homotopy classes in $S_{\varepsilon}(M)$ so that  the dynamical properties of the geodesic flow can be of use. 

The fact that one can compute $E(\bar h)$ follows from \cite{kahn-markovic} and from the work of Uhlenbeck in \cite{uhlenbeck}. The inequality in Theorem \ref{main.thm} is a consequence of Gauss-Bonnet Theorem. The uniqueness statement in Theorem \ref{main.thm} will follow in two steps. First we combine minimal surface theory with the strong rigidity properties of totally geodesic discs proven independently  by Shah \cite{shah} and Ratner \cite{ratner}  to find, for every $v\in T_pM$, a totally geodesic  hyperbolic disc in $(M,g)$  containing $(p,v)$ in its tangent space. This will occupy most of the proof . Then we use the ergodicity of the frame flow due to Brin--Gromov \cite{brin-gromov} to show that the sectional curvature of every plane is $-1$.

We now briefly review some previous results related to our work. 

Shah \cite{shah} and Ratner \cite{ratner} showed that a totally geodesic  immersion of $\H^2$ in a compact hyperbolic  manifold has its image either dense or a closed surface. McMullen--Mohammadi--Oh recently generalized this result to the non-compact case  \cite{mmo}. 

McReynolds and Reid \cite{McR14} showed that arithmetic hyperbolic $3$-manifolds  which have the same (non-empty) set of totally geodesic surfaces are commensurable, i.e., covered by a common closed  $3$-manifold. It is not expected that
the areas of all totally geodesic surfaces will determine  the commensurability class of the arithmetic hyperbolic $3$-manifolds  \cite{LMcM17}. Jung \cite{jung} studied the asymptotic behavior  of the areas of totally geodesic surfaces for some arithmetic hyperbolic 3-manifolds.

Totally geodesic surfaces in hyperbolic manifolds have the attractive feature that they are preserved by the geodesic flow but their existence is not guaranteed. For instance,  there are closed hyperbolic $3$-manifolds which admit no totally geodesic immersed closed surface \cite[Chapter 5.3]{MacR} and even  finite volume hyperbolic
$3$-manifolds which admit no totally geodesic immersed finite area surfaces either \cite{calegari}. Recently it was shown that a closed hyperbolic $3$-manifolds having infinitely many totally geodesic surfaces \cite{bfms,margulis-mohammadi} is arithmetic.

Finally, it was shown in \cite{McR18}  that the commensurability class of a closed hyperbolic $3$-manifold is determined by their surface groups.
\medskip

{\bf Acknowledgements:} The authors would like to thank Ursula Hamenst\"adt for her comments and remarks.

\section{Notation and preliminaries}
We set up the basic notation and then discuss several results all well-known among the experts.

There is a discrete subgroup $\Gamma\subset \text{Isom}^{+}(\H^3)=\text{PSL}(2,\C)$ so that  $M=\H^3\setminus \Gamma$  is a closed orientable $3$-manifold and we fix an isomorphism between $\pi_1(M)$ and $\Gamma$. A Riemannian metric on $M$ is denoted by $h$ and the hyperbolic metric  is denoted by $\bar h$. Geometric quantities with respect to the metric $h$ will usually have the subscript $h$, while the same quantities will have no subscript if computed with respect the metric $\bar h$. For instance, the distance between two points $p,q$, the area of an immersed surface {$\phi:\Sigma\rightarrow M$, or  the Hausdorff distance between sets $A,B$ with respect to the metric $\phi^*\bar h$ and $\phi^*h$}, respectively,  are denoted by $d(p,q), d_h(p,q)$, $\text{area}(\Sigma)$, $\text{area}_h(\Sigma)$ or $d_H(A,B)$, $d_{H,h}(A,B)$. {Note that if $\Sigma$ is a $k$-cover of a surface $\tilde \Sigma$, then $\text{area}_h(\Sigma)=k\,\text{area}_h(\tilde \Sigma)$.}


Let $(B^3,h)$ denote the universal cover of $(M,h)$ and $S^2_{\infty}$ denote its sphere at infinity, which is defined as the set of all asymptote classes of geodesic rays, where two geodesic rays $\gamma_i:[0,+\infty)\rightarrow B^3$, $i=1,2$, define the same asymptote class, denoted by $\gamma_1(+\infty)$, if $\lim_{t\to \infty}d_h(\gamma_1(t),\gamma_2(t))<+\infty$. There is a natural topology on $\overline B^3 :=B^3\cup S^2_{\infty}$, the cone topology (see \cite{anderson82} for instance), for which $\bar B^3$ is homeomorphic to a $3$-ball. Given a set $\Omega\subset B^3$ we denote by $\overline \Omega$ its closure in $\overline B^3$ and $\partial_\infty\Omega$ stands for $\overline \Omega\cap S^2_{\infty}$.
We follow convention and denote $(B^3,\bar h)$ simply by $\H^3$. 


An essential immersion $\phi:\Sigma\rightarrow M$  must have genus $\geq 2$ (by Preissman Theorem) and thus $\phi$ admits a lift  $\bar \phi:D \rightarrow \H^3$ from a disc $D$ onto $\H^3$. To ease notation, we will often identity the immersions of $\Sigma$ or $D$ with its images in $M$ or $\H^3$, respectively. {This will create an ambiguity when $\Sigma$ is a $k$-cover of another surface $\tilde \Sigma$, but it will be clear from the context whether we are referring to the immersion (when we compute area for instance) or to the image set in $M$ (when we compute Hausdorff distances for instance)}. Given an essential surface $\Sigma\subset M$ with surface group $G<\Gamma$, there is a lift $D\subset \H^3$ that is invariant under $G$. Any other disc $D'\subset \H^3$ lifting $\Sigma$ is invariant under a group $G' <\Gamma$ that is conjugate to $G$. Necessarily we have (with an obvious abuse of notation) $D\setminus G=D'\setminus G'=\Sigma$.

The Grassmanian bundle of unoriented $2$-planes  in $M$ or $\H^3$  is denoted by $Gr_2(M)$ or $Gr_2(\H^3)$, respectively.  An immersed surface $\Sigma$  in $M$ (or its lift $D$ in $\H^3$) induces a natural immersion into $Gr_2(M)$ (or $Gr_2(\H^3)$) via the map $p\mapsto (p,T_p\Sigma)$ (or $p\mapsto (p,T_pD)$). 


\subsection{Fundamental domains and Cayley graphs:}\label{fundamental.domain} Given a subgroup $G <\text{PSL}(2,\C)$ acting properly discontinuous on $\H^3$, a {\em fundamental domain} $\Delta\subset \H^3$ for $\H^3\setminus G$ is  a closed region so that 
\begin{itemize}
\item[(i)] $\cup_{\phi\in G}\phi(\Delta)=\H^3$;
\item[(ii)] $\phi\in G$ and $\phi(\Delta)\cap \text{int }\Delta\neq \emptyset\implies \phi=\text{Id}$.
\end{itemize}
Because the manifold $M$ is compact, we can choose its fundamental domain   $\Delta$ to be  a convex polyhedron with finitely many totally geodesic faces. Such domains are called {\em Dirichlet  fundamental domain}. 
Each compact set $K\subset \H^3$ intersects only finitely many elements of $\{\phi(\Delta)\}_{\phi\in \Gamma}$.

Given a subgroup $G< \Gamma$, we  consider the set  $\Gamma\setminus G=\{\phi G: \phi\in \Gamma\}$ and pick a representative $\underline\phi$ in each coset $\phi G$. 

\begin{lemm}
 $\Delta_G=\cup_{\underline \phi\in \Gamma\setminus G}\underline\phi^{-1}(\Delta)$ is a fundamental domain for $\H^3\setminus G$. 
\end{lemm}
\begin{proof}
The reader can check that $\Delta_G$ is closed and that $\cup_{\phi\in G}\phi(\Delta_G)=\H^3$.  Suppose there is $\psi\in G$ and $x\in \psi(\Delta_G)\cap \text{int }\Delta_G$. Because $x\in  \text{int }\Delta_G$ we can find a finite set $A\subset \Gamma\setminus G$ and an open set $U$ so that $x\in U\subset \cup_{\underline \phi\in A}\underline \phi^{-1}(\Delta)$. Likewise we have $x\in \psi(\underline\sigma^{-1}(\Delta))$ for some $\underline \sigma\in \Gamma\setminus G$.  We must have
$$ \psi(\underline\sigma^{-1}(\text{int }\Delta))\cap \left( \cup_{\underline \phi\in A}\underline \phi^{-1}(\Delta)\right)\neq \emptyset$$
and thus $(\underline \sigma \psi^{-1})^{-1}=\underline \phi^{-1}$ for some $\underline \phi\in A$. Hence $ \underline \sigma=\underline \phi$ and $\psi=\text{Id}$.
\end{proof}

Fix $p\in \H^3$. Choosing $R$ large enough, the set $A=\{\phi\in\Gamma: d(p,\phi(p))\leq R\}$ generates $\Gamma$. The {\em Cayley graph} $\text{Gr}(\Gamma,A)$ of $\Gamma$ generated by $A$ is defined as having  vertices $\{\phi(p)\}_{\phi\in\Gamma}$ and two vertices $\psi(p),\phi(p)$ are connected by an edge if $\phi\psi^{-1}\in A$.  The graph $\text{Gr}(\Gamma,A)$ admits a distance function $\mathcal d$, where $\mathcal d(\phi,\psi)$ is the word length of $\phi\psi^{-1}$, and the norm of $\phi\in\Gamma$ is given by $|\phi|=\mathcal d(\phi,\text{Id})$. The Hausdorff distance with respect to $\mathcal d$ is denoted by $\mathcal d_H$.
We will need the following lemma
\begin{lemm}\label{cayley.balls} There is a constant $c>0$ depending on the Dirichlet domain $\Delta$ containing $p$ so that
\begin{equation*}
B_{n/c-c}(p)\subset\cup_{|\phi|\leq n}\phi(\Delta)\subset B_{nc+c}(p)\quad\text{for all }n\in\N.
\end{equation*}
\end{lemm}
\begin{proof}
The \v Svarc-Milnor Lemma says that 
the map $\Gamma\rightarrow \H^3$, $\phi\mapsto \phi(p)$ is a quasi-isometry  meaning there is a constant $K$ so that 
\begin{itemize}
\item[i)]$\H^3=\cup_{\phi\in\Gamma}B_K(\phi(p))$;
\item[ii)] for all $\psi,\phi\in\Gamma$, $$K^{-1}d(\phi(p),\psi(p))-K\leq \mathcal d(\phi,\psi)\leq Kd(\phi(p),\psi(p))+K$$
\end{itemize}
and there are constants $n_1\in\N$, $K_1>0$  so that 
$B_{K_1}(p)\subset \cup_{|\phi|\leq n_1}\phi(\Delta)$ and $\Delta\subset B_{K_1}(p)$.  The constant $c$ can be computed in terms of $n_1,K,K_1$ and we leave it to the reader.
\end{proof}
The \v Svarc-Milnor Lemma mentioned above also says that choice of a generating set or different base points would  give another Cayley graph that is quasi-isometric  to $\text{Gr}(\Gamma,A)$. We abuse notation and simply denote the Cayley graph by $\Gamma$.

\subsection{Morse Lemma: }
A curve $\gamma:\R\rightarrow \H^3$ is a {\em $(K,c)$ quasi-geodesic} if 
$$K^{-1}d(\gamma(t),\gamma(s))-c\leq |t-s|\leq Kd(\gamma(t),\gamma(s))+c\quad\text{for all }s,t\in\R.$$
A geodesic in $B^3$ with respect to the metric $h$ is a $(K,0)$ quasi-geodesic for some $K=K(h)$.

Morse's Lemma (see \cite[Theorem 2.3]{knieper} for instance) gives the existence of $r_0=r_0(K,c)$ such that for every  $(K,c)$ quasi-geodesic $\gamma$ in $\H^3$ there is a unique (up to reparametrization) geodesic $\sigma$ in $\H^3$ so that the Hausdorff distance between $\sigma(\R)$ and $\gamma(\R)$ is bounded by $r_0$.

\subsection{Limit sets and quasi-Fuchsian manifolds:}\label{limit.set.section}
 
 Given a discrete subgroup acting properly discontinuously $G< \text{PSL}(2,\C)$, the {\em limit set} $\Lambda(G)\subset S^2_{\infty}$ is defined as being the set of accumulation points in $S^2_{\infty}$ of the orbit $Gx$, where $x\in\H^3$. It is well known that the definition is independent of the point $x\in\H^3$ chosen and that the limit set is  closed. Elements $\phi\in \text{PSL}(2,\C)$ induce conformal maps of $S^2_{\infty}$ and one has that $\Lambda(\phi G\phi^{-1})=\phi(\Lambda(G))$. From this one deduces $\Lambda(G)$ is $G$-invariant.
 
 A $C^1$-map  $F:S^2_{\infty}\rightarrow S^2_{\infty}$ is said to be {\em quasiconformal} with dilation bounded by some $K\in[1,+\infty)$ if for for every point $p$, $DF_p$  sends circles into ellipses whose eccentricity (ratio between major axis and minor axis) is bounded by $K$. Conformal maps are quasi-conformal maps with $K=1$.
 
  If $\Lambda(G)$ is a geometric circle, then $G$ is called a {\em Fuchsian} group and if $\Lambda(G)$ is a Jordan curve, then  $G$ is called a {\em quasi-Fuchsian} group. In this case, it is known \cite[Proposition 8.7.2]{thurston} that  $\Lambda(G)$ is a $K$-quasicircle, meaning there is a quasiconformal map $F$ with dilation bounded by $K$ that  maps the equator to $\Lambda(G)$. 
  
 \begin{defi}\label{limit.set.definition}
  A discrete subgroup acting properly discontinuously $G< \text{PSL}(2,\C)$ is {\em $\varepsilon$-Fuchsian} if $\Lambda(G)$ is a $(1+\varepsilon)$-quasicircle.  This notion is invariant under conjugacy.
  \end{defi}


The normal bundle of an orientable surface $S\subset M$ is denoted by $T^{\bot}S\simeq S\times \R$. When the background metric is the hyperbolic metric, Uhlenbeck proved the following result in \cite[Theorem 3.3]{uhlenbeck}. 

\begin{thm}\label{uhlenbeck.thm}
Let  $S\subset M$ be an orientable minimal surface with principal curvatures $|\lambda(x)|\leq \lambda_0\leq 1$ for all $x\in S$. Then
\begin{itemize}
\item[(i)]The exponential map $\mbox{exp}:T^{\bot}S\rightarrow M$ is a covering map and thus $G:=\text{exp}_*(\pi_1(S))$ is a surface group;
\item[(ii)] $G$ is a quasi-Fuchsian group and $N:=\H^3\setminus G\simeq T^{\bot}S $ is complete hyperbolic manifold;
\item[(iii)] $S$ is embedded, area-minimizing, and  the only closed minimal surface in $N$;
\item[(iv)] For all $t>\tanh^{-1}(\lambda_0)$, the region $S\times[-t,t]\subset N$ is strictly convex and is boundary  has principal curvatures bounded from above  by 
$$\frac{\sinh t+\cosh t \lambda_0}{\cosh t+\sinh t \lambda_0}.$$
\end{itemize}
\end{thm}
The last property is not explicitly stated in \cite[Theorem 3.3]{uhlenbeck} but from its proof one sees that the surface $S\times\{t\}\subset N$ has principal curvatures
$$\lambda^t_{\pm}(x) =\frac{\sinh t\pm\cosh t \lambda(x)}{\cosh t\pm\sinh t \lambda(x)},$$
which readily implies property (iv).

\subsection{Totally geodesic planes}\label{subsection.geodesic.planes}

Consider $\mathcal C$ to the set of all geometric circles (of varying radii) in $S^2_{\infty}$. This set is non-compact and in one-to-one correspondence with the totally geodesic discs in $\H^3$ because given any $\gamma\in\mathcal C$ there is exactly one totally geodesic disc $C(\gamma)\subset \H^3$ such that $\partial_{\infty}C(\gamma):=S^2_{\infty}\cap \overline{C(\gamma)}$ is identical to $\gamma$.

Every $\phi\in \text{PSL}(2,\C)$  induces  a map from $\mathcal C$ to $\mathcal C$ (still denoted by $\phi$) such that $\phi(C(\gamma))=C(\phi(\gamma)).$ Hence, the group  $\Gamma$ acts naturally on $\mathcal C$. 

The following result was proven independently by Ratner \cite{ratner} and Shah \cite{shah}.
\begin{thm}\label{ratner-shah}
Given $\gamma\in \mathcal C$,  either  $C(\gamma)$ projects to a closed surface in $M$ or its natural immersion into $Gr_2(\H^3)$ projects to a dense set in $Gr_2(M)$. 
\end{thm}
Given $\gamma\in\mathcal C$, consider the orbit
$ \Gamma \gamma:=\{\phi(\gamma): \phi\in \Gamma\}\subset \mathcal C.$

Using the fact that  $\{\gamma_i\}_{i\in\N} \in \mathcal C$ converges to $\gamma\in \mathcal C$ if and only if $C(\gamma_i)$ converges to $C(\gamma)$ on compact sets of $\H^3$, we leave to the reader to check that  $\Gamma \gamma$
is dense in $\mathcal C$ if and only if the natural immersion of $C(\gamma)$ into $Gr_2(\H^3)$ projects to a dense set in $Gr_2(M)$.

 The next theorem was essentially proven Theorem 11.1 in \cite{mmo}. We provide the modifications that need to be made.

\begin{thm}\label{mmo-lemma} Consider $\mathcal L\subset\mathcal C$ a closed set that is $\Gamma$-invariant.

Suppose that no element in $\mathcal L$ has a dense $\Gamma-$orbit in $\mathcal C$. Then every $\gamma\in \mathcal L$ is isolated and has $C(\gamma)$ projecting to a closed surface in $M$.\end{thm}

\begin{proof}
Every $\gamma\in \mathcal L$ must have $C(\gamma)$ projecting to a closed  surface in $M$  because otherwise Theorem \ref{ratner-shah} would say that $\Gamma\gamma$ is dense in $\mathcal C$.

 We argue by contradiction and suppose there is $\gamma_i\in\mathcal L$ converging to $\gamma$ in $\mathcal L$ as $i\to\infty$ with $\gamma_i\neq \gamma$. Set $$\Gamma^\gamma=\{\phi\in\Gamma: \phi(\gamma)=\gamma\}.$$
The action of $\Gamma^\gamma$ preserves $C(\gamma)$ and  $C(\gamma)\setminus \Gamma^\gamma$ corresponds to a closed surface because $C(\gamma)$  projects to a closed surface in $M$.

Choose a disc $\Omega\subset S^2_{\infty}$ so that $\partial \Omega=\gamma$.  Either $\Gamma^\gamma$ preserves $\Omega$ or it contains a normal subgroup of index $2$ that preserves $\Omega$. If the latter occurs, relabel $\Gamma^\gamma$ to be that subgroup. By swapping $\Omega$ with its complement in $S^2_{\infty}$ if necessary and after possibly passing to a subsequence of $\{\gamma_i\}_{i\in\N}$, we can assume that $\gamma_i\cap \Omega\neq \emptyset$ for all $i\in\N$.

The disc $D$ carries a natural hyperbolic metric $h_{\Omega}$ conformal to the round metric in $S^2_{\infty}$  and each map in $\Gamma^{\gamma}$ is an orientation preserving  isometry of $\Omega$ with respect to the metric $h_\Omega$.  Finally, $\Omega\setminus \Gamma^{\gamma}$ is isometric to $C(\gamma)\setminus \Gamma^{\gamma}$ and so the group $\Gamma^{\gamma}$ is a nonelementary, convex, cocompact Fuchsian group as defined in \cite[Section 3]{mmo}.  Hence we can apply  Corollary 3.2 of \cite{mmo}  which says that if we consider  the set  $\mathcal H(D)
$ of all  horocycles in $(\Omega, h_{\Omega})$, i.e., 
$$\mathcal H(\Omega)=\{\sigma\in\mathcal C: \sigma\subset \overline \Omega, \sigma\cap \partial \Omega\neq \emptyset\},$$
then  the closure of $\bigcup \Gamma^\gamma \gamma_{i}$, and hence $\mathcal L$, contains $\mathcal H(\Omega)$.

From Theorem 4.1\cite{mmo} there exists a dense set $\Lambda_0\subset S^2_{\infty}$ such that if $\sigma\in \mathcal C$ intersects $\Lambda_0$, then $\Gamma \sigma$ is dense in $\mathcal C$. Necessarily, $\Lambda_0$ must intersect some element of $\mathcal H(\Omega)$ and so there is $\sigma\in\mathcal L$ for which $\Gamma\sigma$ is dense in $\mathcal C$. Thus $C(\sigma)$ does not project to a closed surface in $M$, which is a contradiction.

\end{proof}

\subsection{Frame flow}

We denote the  bundle of oriented orthonormal frames of $M$ with respect to $\bar h$ or $h$  by $\mathcal F(M)$ and $\mathcal F(M)(h)$ respectively.

The {\em frame flow} $F_t:\mathcal F(M)(h)\rightarrow \mathcal F(M)(h)$ is defined in the following way: given an oriented frame $(e_1,e_2,e_3)$ for $T_p M$, 
$$F_t(p,(e_1,e_2,e_3))=(\gamma(t),(\gamma'(t), e_2(t),e_3(t)))$$
where $\gamma(t)=\exp_p(te_1)$, and $e_2(t),e_3(t)$ denote the parallel transport of $e_2,e_3$ along $\gamma$. An important result which we will use, due to Brin and Gromov \cite{brin-gromov}, says that when $(M,h)$ is negatively curved  the frame flow is ergodic and in particular has a dense orbit in $\mathcal F(M)(h)$.

\section{Convex hulls} In this section we assume  that  $(M,h)$ has sectional curvature less than or equal to $-1$.

Given a closed set $\Lambda\subset S^2_{\infty}$, its {\em convex hull} $C_h(\Lambda)\subset \bar B^3$ denotes the smallest geodesically closed set of  $\bar B^3$ (with respect to the metric $h$) that contains $\Lambda$.  

The goal of this section is to prove the following result.

\begin{thm}\label{convex.hull.main.thm}
Let  $S\subset M$ be a minimal surface (with respect to $\bar h$) with principal curvatures $|\lambda(x)|\leq \lambda_0< 1$ for all $x\in S$ and $\Sigma\subset M$ a minimal surface with respect to $h$ in the homotopy class of $S$. Then, denoting by $D, \Omega\subset \H^3$ the lifts  of $S$ and $\Sigma$, respectively, that are invariant by the same surface group  we have
$$d_H(D,\Omega)\leq R$$
for some constant $R=R(h,\lambda_0)$.
\end{thm}


Bangert and Lang proved similar results to the theorem above (see \cite{bangert-lang} and references therein) under the conditions that $D$ and $\Omega$ are quasi-minimizing. While that will be true for $D$, it is not necessarily true for $\Omega$ and so the result cannot be straightforwardly applied. It is conceivable that their proof could be extended to our setting but we chose a different argument.

 Given $p\in B^3$, the cone over $\Lambda$ centered at $p$  with respect to the metric $h$ is given by
$$\text{Co}_p(\Lambda):=\text{clo} \{\gamma(t): \gamma \mbox{ a  geodesic with }\gamma(0)=p, \gamma(\infty)\in\Lambda, 0\leq t<\infty\},$$
where the closure is taken with respect to the cone topology. One has $\text{Co}_p(\Lambda)\cap S^2_{\infty}=\Lambda$.

The space  $(B^3,h)$  has sectional curvature less than or equal to $-1$ and is thus $\bar \delta$-hyperbolic for some universal constant $\bar\delta$, meaning that a side in any geodesic triangle (with vertices possibly in $S^2_{\infty}$) is contained in the  $\bar \delta$-neighborhood of the union of the other two sides.  Thus if $p, q\in B^3$, $x\in S^2_{\infty}$, and $\gamma, \sigma$ denote   geodesic rays (with respect to $h$) staring at $p, q$ respectively with $\gamma(\infty)=\sigma(\infty)=x$ then, with $l$ denoting the geodesic connecting $p$ to $q$, we have that $\gamma$ is contained in the $\bar \delta$-neighborhood of the union of $\sigma$ and $l$. Therefore
\begin{equation}\label{cone.distance}
d_{H,h}(\text{Co}_p(\Lambda), \text{Co}_q(\Lambda))\leq \bar \delta+d_h(p,q).
\end{equation}
Likewise, $\text{Co}_p(\Lambda)$ is $\bar\delta$-quasiconvex, meaning that given any $x,y$ in  $\text{Co}_p(\Lambda)$, the geodesic connecting $x$ to $y$ is contained in a $\bar\delta$-neighborhood of $\text{Co}_p(\Lambda)$. 

\begin{prop}\label{cor.convex.hull} There is $R=R(h)$ so that given a closed set $\Lambda\subset S^2_{\infty}$ then  
$C_h(\Lambda)\cap S^2_{\infty}=\Lambda$ and
$$d_H(C_h(\Lambda), C_{\bar h}(\Lambda))\leq R.$$
\end{prop}
\begin{proof}
The key step in the proof is the following claim:
\medskip

\noindent{{\bf Claim 1}:}{\em There is $R=R(h)$ so that for every $p\in C_h(\Lambda)$,  $$d_{H,h}(C_h(\Lambda), \text{Co}_p(\Lambda))\leq R.$$
In particular, $C_h(\Lambda)\cap S^2_{\infty}=\Lambda$.}
\medskip

We have at once that $\text{Co}_p(\Lambda)\subset C_h(\Lambda)$. 
In Proposition 2.5.4 of \cite{bowditch}, Bowditch used the existence of certain convex sets constructed by Anderson  in \cite{anderson83} and the fact that $\text{Co}_p(\Lambda)$ is $\bar\delta$-quasiconvex to show  the existence of a constant $R=R(h)$  so that $C_h(\Lambda)$ is contained in a $R$-neighborhood of $\text{Co}_p(\Lambda)$. This implies the claim.

If  $\gamma,\bar \gamma$ are two geodesics with respect to $h$ and $\bar h$ respectively that  connect $p\in \H^3$ (or $y\in S^2_{\infty}$) to $x\in S^2_{\infty}$, Morse's Lemma gives the existence of a constant $r_0$ depending only on $h$  so that  $d_H(\gamma,\bar\gamma)\leq r_0$. From this we deduce that 
$$\text{dist}(C_h(\Lambda), C_{\bar h}(\Lambda))\leq r_0\quad\text{and}\quad d_H(\text{Co}_p(\Lambda,h),\text{Co}_p(\Lambda,\bar h))\leq r_0,$$
where $\text{Co}_p(\Lambda,h)$ denotes the cone with respect to $h$. Combining these inequalities with \eqref{cone.distance} and Claim 1 we deduce the desired result at once.
\end{proof}

Let $G$ be a quasi-Fuchsian surface group and set $N:=\H^3\setminus G$.  Because $\Lambda(G)\subset S^2_{\infty}$ is $G$-invariant, $C_h(\Lambda(G))$ is also $G$-invariant and    $C_h(N):=C_h(\Lambda(G))\setminus G$  is a compact subset of the  $N$ {(see \cite[Section 8.2]{thurston}).}
 \begin{prop}\label{convex.core}  Every closed immersed minimal surface in $(N,h)$ is contained in $C_h(N)$.
\end{prop}
\begin{proof}
Let $\tilde d:N\rightarrow [0,\infty)$ be the distance function to $C_h(N)$. If $\pi$ denotes the projection from $(B^3,h)$ to $(N,h)$, we have that $\pi^{-1}(C_h(N))= C_h(\Lambda(G))$ is a geodesically convex set and so  Proposition 4.7  in \cite{bishop} says that $\tilde d$ is a  continuous convex function (Theorem 4.7 \cite{bishop} is misstated because it requires the subset of $N$ to be geodesically convex instead of requiring the inverse image of the set  under the covering map to be geodesically convex).

Given $\Sigma$ a closed connected minimal immersion, there is $l>0$ so that $\Sigma\subset \tilde d^{-1}[0,l)$ and set $K=\tilde d^{-1}[0,l+1]$.

The function $\tilde d$ does not have to be smooth but we can apply \cite[Theorem 2]{greene} to obtain a sequence of smooth functions $\{\phi_i\}_{i\in\N}$ so that $\phi_i$ tends to $\tilde d$ uniformly in $K$ as $i\to\infty$ and, setting
$$\lambda(\phi_i)=\min\{D^2\phi_i(v,v): x\in K, v\in T_xN, |v|=1\},$$
we have 
$\liminf_{i\to\infty} \lambda(\phi_i)\geq 0$.  Hence  $\Delta_{\Sigma}\phi_i\geq \lambda(\phi_i)$ on $\Sigma$ because $\Sigma$ is a minimal surface.

Set $\phi_i^{+}=\max\{\phi_i,0\}$. We have
$$\int_{\{x\in \Sigma:\phi_i\geq 0\}}|\nabla \phi_i|^2dA_h=-\int_{\Sigma}\phi_i^{+}\Delta\phi_idA_h\leq -\lambda(\phi_i)\int_{\Sigma}\phi_i^{+} dA_h$$
and so 
$$\lim_{i\to\infty} \int_{\{x\in \Sigma:\phi_i\geq 0\}}|\nabla \phi_i|^2dA_h=0.$$

Suppose that $\Sigma\cap \tilde d^{-1}\{\delta\}\neq \emptyset$ for some $l>\delta>0$. Note that $\tilde d\in W^{1,2}(\Sigma)$, the functions  $\phi_i$ converges weakly  to $\tilde d$ in $W^{1,2}(\Sigma)$ as $i\to\infty$, and so
$$\int_{\tilde d^{-1}[\delta,l]\cap \Sigma}|\nabla d|^2dA_h\leq \liminf_{i\to\infty}\int_{\{x\in \Sigma:\phi_i\geq 0\}}|\nabla \phi_i|^2dA_h=0.$$
Thus there is some $t\geq \delta$ so that $\Sigma\subset \tilde d^{-1}(t)$. {An inspection of the proof of Proposition 4.7 of \cite{bishop} shows that $\{\tilde d\leq t\}$ is actually geodesically strictly convex because the ambient curvature is strictly negative and so it cannot contain the minimal surface $\Sigma$ in its boundary $\partial\{\tilde d\leq t\}=\tilde d^{-1}(t)$.}
\end{proof}

\begin{proof}[Proof of Theorem \ref{convex.hull.main.thm}]
Without loss of generality we can assume that $S$ is orientable.

Let $G$ be the surface group that preserves both $D$ and $\Omega$ so that $S=D\setminus G$ and $\Sigma=\Omega\setminus G$. Set $\Lambda$ to be the Jordan curve $\Lambda(G)$. From Theorem \ref{uhlenbeck.thm} there is $\bar t=\bar t(\lambda_0)$ so that
\begin{equation}\label{hausdorff.convex.hull}
d_H(C_{\bar h}(\Lambda),D)\leq \bar t
\end{equation}
and,  for all $x\in D$, if $\gamma_x$ denotes the unit speed hyperbolic geodesic with  $\gamma_x(0)=x$ and $\gamma'_x(0)$ orthogonal to $T_xD$, we have
\begin{equation}\label{geod.away}
\text{dist}(\gamma(t), C_{\bar h}(\Lambda))\geq R+1\quad\text{for all }|t|\geq \bar t+R,
\end{equation} 
where $R=R(h)$ is the constant given by Proposition \ref{cor.convex.hull}.

From Proposition \ref{convex.core} we have that  $\Omega\subset C_{h}(\Lambda)$ and thus we obtain from \eqref{hausdorff.convex.hull} and Corollary \ref{cor.convex.hull} that $\Omega$ is contained in the $\bar t+R$-neighborhood of $D$.

To deduce the other inclusion pick $x\in D$. We have that $\gamma_x(+\infty), \gamma_x(-\infty)$ lie in different connected components of $S^2_{\infty}\setminus \Lambda$. Because $\overline \Omega\subset \overline B^3$ is a disc with the same boundary as $\overline D$, $\gamma_x$ must intersect $\Omega$ in at least one point $\gamma(t)\in \Omega\cap\gamma$. From \eqref{geod.away}  and Proposition \ref{cor.convex.hull} we have that $|t|\leq \bar t+R$  and so $d(x,\Omega)\leq \bar t+R$. 
\end{proof}

\section{Almost-Fuchsian surface groups}\label{black.box}

Let  $s(M,g,\varepsilon)$ denote the cardinality of  $S(M,g,\varepsilon)$, the set of $\varepsilon$-Fuchsian surface subgroups of genus at most $g$, modulo the equivalence relation of  conjugacy. Recall that we defined $S_{\varepsilon}(M)=\cup_{g\in\N}S(M,g,\varepsilon)$.

\begin{prop}\label{compactness-thm} Suppose we have a sequence $\Pi_i\in S_{\delta_i}(M)$, where $\delta_i\to 0$ as $i\to\infty$. For each $i\in\N$, there is an essential minimal surface $S_i$ in the homotopy class $\Pi_i$ so that $\area(S_i)=\area(\Pi_i)$ and  
\begin{equation}\label{seppi.A}
\lim_{i\to\infty}||A||^2_{L^{\infty}(S_i)}=0.
\end{equation}
Moreover, if $D_i$ is a disc lifting  $S_i$ to $\H^3$ that is  preserved by the surface group $G_i<\Gamma$ induced by $S_i$ and intersecting a fixed compact set in $\H^3$ for all $i\in\N$, there  is a totally geodesic disc $D\subset \H^3$ such that, after passing to a subsequence, $D_i$ converges smoothly to $D$ on compact sets and $\Lambda(G_i)$ converges in Hausdorff distance to $\partial_{\infty} D$ in $S^2_{\infty}$.
\end{prop}
\begin{proof}
For each $i\in\N$, consider the essential immersion $S_i\subset M$ that minimizes area with respect to the hyperbolic metric in the homotopy class $\Pi_i$ (using \cite{schoen-yau} for instance). If $D_i$ is the minimal disc lifting  $S_i$ to $\H^3$ that is preserved by the surface group $G_i<\Gamma$ induced by $S_i$ then we have from Theorem 1 in \cite{seppi} that $||A||^2_{L^{\infty}(D_i)}$ tends to zero as $i\to\infty$. The author achieves this by  showing  that the convex hull of $\Lambda(G_i)$ is contained, for all $i$ sufficiently large, between two equidistant disks whose principal curvatures are arbitrarily small, have the same boundary at $S^2_{\infty}$, and are arbitrarily close to each other in Hausdorff distance. 

Assume that all the discs $D_i$  intersect a compact set. We now argue that, after passing to a subsequence, the discs $D_i$ converge to a totally geodesic disc with multiplicity one. From Theorem \ref{uhlenbeck.thm} we know that, for all $i$ sufficiently large, $D_i$ is embedded,  $S_i$ is the unique closed  embedded minimal surface in $M_i=\H^3\setminus G_i\simeq T^{\bot}S_i$  and so area-minimizing in $M_i$  among all mod 2 cycles representing the same element in $H_2(M_i;\Z_2).$ As a result, $D_i$ is locally area-minimizing among mod 2 cycles as well. 
Pick $p_i\in D_i$ which converges, after passing to a subsequence, to some $p\in \H^3$. From the fact that for all $i$ sufficiently large, the embedded discs $D_i$ are locally area-minimizing  among mod 2 cycles, we obtain from  standard compactness theory for minimal surfaces   the existence of a totally geodesic disc $D\subset \H^3$ containing $p$ such that, after passing to another subsequence, $D_i$ converges graphically to $D$ on compact sets.

Consider  $q_i\in\Lambda(G_i)$,  $\sigma_i\subset \H^3$  the geodesic ray with $\sigma_i(0)=p_i$, $\sigma_i(+\infty)=q_i$, and $\gamma_i\subset D_i$ the  geodesic ray (for the induced metric on $D_i$) with $\gamma_i(0)=p_i$, $\gamma_i(+\infty)=q_i$.  The geodesic curvature of $\gamma_i$ in $\H^3$ is a fixed amount below $1$ for all $i$ sufficiently large and so, using tubular neighborhoods  of $\sigma_i$ as barriers, we deduce the existence of $r>0$ so that  $\gamma_i$ is contained in a $r$-tubular neighborhood of $\sigma_i$  for all $i\in\N$. Thus, after  passing to a subsequence, both curves converge on compact sets to the same geodesic ray $\sigma\subset D$. Using this fact, the reader can deduce that  $\Lambda(G_i)$ converges in Hausdorff distance to $\partial_{\infty} D$ in $S^2_{\infty}.$
\end{proof}
Using the above proposition we now show the following improvement to the main results of \cite{kahn-markovic2,kahn-markovic}.

\begin{thm}\label{counting.thm} There are positive constants $c_1=c_1(M,\varepsilon)$, $c_2=c_2(M)$, and $k=k(M,\varepsilon)\in\N$ so that  for all $g\geq k$ we have
$$(c_1g)^{2g} \leq s(M,g,\varepsilon) \leq  (c_2g)^{2g}.
 $$
 Moreover, there is a subset $G(M,g,\varepsilon)\subset S(M,g,\varepsilon)$ with more than $(c_1g)^{2g}$ elements so that any sequence  of homotopy classes $\Pi_i\in G(M,g_i,1/i)$, $i\in\N$ has a representative {$\phi:S_i\rightarrow M$} so that 
 \begin{itemize}
 \item[(a)] $S_i$ is a minimal {immersion} with $\area(S_i)=\area(\Pi_i)$ and $$\lim_{i\to\infty}\sup_{S_i}|A|=0;$$
 \item[(b)] after passing to a subsequence, the Radon measure 
 {$$f\in C^0(M)\mapsto \mu_i(f)=\frac{1}{\text{area}(S_i)}\int_{S_i}f\circ \phi dA$$}
 converges to a measure $\nu$  which is positive on every open set of $M$.
 \end{itemize}
\end{thm}
\begin{proof}
If $s(M,g)$ denotes the cardinality of $S(M,g)$, it was shown in Theorem 1.1 from \cite{kahn-markovic} the existence of  $c_2>0$ so that $s(M,g) \le (c_2g)^{2g}$ for all $g$ large.  Since $s(M,g,\varepsilon)\le s(M,g)$, the upper bound is verified.

We now verify the lower bound.  In \cite{kahn-markovic2} the authors show that for all $\varepsilon>0$ there is a Fuchsian group $K$ (preserving a totally geodesic plane $C(\gamma)$ for some geodesic circle $\gamma$) and a $(1+\varepsilon)$-quasiconformal map $\Phi:S^2_{\infty}\rightarrow S^2_{\infty}$ so that $G=\Phi\circ K\circ \Phi^{-1}$ is a surface subgroup of $\Gamma$. The map $\Phi$ admits an extension $F:\mathbb{H}^3\rightarrow \H^3$ that is equivariant with respect to $K$ and $G$ and a $(1+o_{\varepsilon}(1), o_{\varepsilon}(1))$- quasi-isometry, where  $o_{\varepsilon}(1)$ denotes a quantity depending only on $M$ and $\varepsilon$ that tends to zero as $\varepsilon\to 0$. As a result, the essential surface  $\Sigma_\varepsilon=F(C(\gamma)\setminus K)\subset M$ induces an element of $S_{\varepsilon}(M)$. $\Sigma_\varepsilon$ has the property that geodesics with respect to the intrinsic distance are $(1+\varepsilon,\varepsilon)$-quasigeodesics and we denote such surfaces  by {\em $(1+\varepsilon)$-quasigeodesic surfaces}.

Let $g_0$ denote the genus of $\Sigma_\varepsilon$.  If $\Sigma_n$ denotes a degree $n$ cover of $\Sigma_\varepsilon$ then its genus  is $g=n(g_0-1)+1$ and so $\Sigma_n$ induces an element in $s(M,g,\varepsilon)$.  The  M\"uller-Puchta's formula says that the number of index $n$ subgroups of a genus
$g_0$ orientable surface  grows like $2n(n!)^{2g_0-2}(1+o(1))$ and so (using Stirling's approximation) we get the estimate $$s(M,g,\varepsilon) \ge (c_1 g)^{2g}$$
where $c_1>0$ depends on $g_0$ which in turn depends only on $M$ and $\varepsilon$.

We set $G(M,g,\varepsilon)$ to be the homotopy classes that come from finite covers of $\Sigma_\varepsilon$ and have genus less than or equal to $g$. We now describe in more detail the properties of  $\Sigma_\varepsilon$ as they will be needed to show that item (a) and (b)  of the theorem hold. 

In \cite{hamenstadt} Hamenstadt extended the results of \cite{kahn-markovic2} to some rank one locally symmetric spaces, building on the work of Khan-Markovic. We follow the geometric description  and the notation of  \cite{hamenstadt}.


The basic building blocks are called $(R,\delta)$-geometric skew-pants $P$ (or simply just geometric skew-pants) and they are defined in Section 4 and Section 6 of \cite{hamenstadt}.  The boundary of  $P$ consists of $3$ closed geodesics in $M$ and $P$ decomposes into  $5$ polygon regions (two center triangles and three twisted bands using the notation in \cite[Section 6]{hamenstadt}) with geodesic boundary.  Each polygon is a smooth immersion whose principal curvatures depend uniformly on $(R,\delta)$ and can be made arbitrarily small by choosing $R$ sufficiently large and $\delta$ sufficiently small. Regions that share  a common geodesic side have the property that the corresponding co-normals make an angle  as close to  $\pi$ as desired by choosing $R$ large and $\delta$ small.  Given any $0<\eta<1$ there is $d>0$ (independent of $R$ and $\delta$) so that the set of points $K^P$ in $P$ that are at an intrinsic distance less than or equal to $d$ from one of the center triangles has $(1-\eta)2\pi\leq \area(K^P)\leq \area(P)\leq (1+\eta)2\pi$ for all $R$ large and  $\delta$ small.  The seams of  a geometric skew-pants $P$ are three shortest geodesic arcs in $M$ (in the homotopy class defined by $P$) that connect  the three boundary geodesics of $P$.

 The quasigeodesic surface $\Sigma_{\varepsilon}$ is constructed by attaching several pairs of $(R,\delta)$-geometric skew-pants  $P, P'$ (see \cite[Definition 6.1]{hamenstadt} for the rigorous definition) so that they share a common boundary geodesic  $\beta$ (with opposite induced orientations),  the tangent planes of $P$ and $P'$ along $\beta$  can be made uniformly close to each other as $R\to\infty$ and $\delta\to 0$, and the endpoints of seams at $P$  in $\beta$ are at a fixed distance from the endpoints of seams at $P'$. This last property is important to ensure that a surface constructed this way will be $(1+\varepsilon)$-quasigeodesic if $R$ and $\delta$ are respectively very large and very small  (see \cite[Proposition 6.2]{hamenstadt}). Note that necessarily $\area(\Sigma_{\varepsilon})\simeq 4\pi(g-1)$ where $g$ is the genus of $\Sigma_{\varepsilon}$ and that in \cite[Lemma 3.1]{hamenstadt} it is shown that $\Sigma_{\varepsilon}$ is a locally CAT$(-1/2)$ space for all $\varepsilon$ sufficiently small.

 With  $0<\eta<1$ fixed, consider the set of points $K$ in $\Sigma_{\varepsilon}$ that are at an intrinsic distance less than or equal to $d$ (given as before) from any of the center triangles coming from the geometric skew-pants. We have  
\begin{equation}\label{eta.constant.lemma}
\area(K)\geq (1-\eta)(1+\eta)^{-1}\area(\Sigma_{\varepsilon})
\end{equation}
  and there is $\underline r>0$ so that, for all $R$ large and $\delta$ small, any intrinsic ball $\hat B_{\underline r}(x)$ in $\Sigma_{\varepsilon}$ of radius ${\underline r}$ centered at $x\in K$ intersects at most a finite number of the polygonal regions with geodesic boundary. In particular, by making $R$ large and $\delta$ small,  we have  $\hat B_{\underline r}(x)$ arbitrarily close to a totally geodesic disc for all $x\in K$. Call that region $K_{\eta}$.

Consider the minimal representatives $S_i$ in the homotopy class $\Pi_i\in G(M,g_i,1/i)$, $i\in\N$, given by Proposition \ref{compactness-thm}. Each $S_i$ is  homotopic to a $(1+1/i)$-quasigeodesic surface $\Sigma_i$ and choose discs $D_i, \Omega_i\subset \H^3$ that cover $S_i$ and $\Sigma_i$, respectively, and such that $\partial_{\infty}  D_i=\partial_{\infty}\Omega_i$. For all $i$ sufficiently large, $\Omega_i$ is a CAT$(-1/2)$ space \cite[Theorem II.4.1]{bridson} for which every geodesic arc can be extended \cite[Proposition II.5.10]{bridson}. Combining with the fact that  the principal curvatures of $S_i$ tend to zero  and that geodesics in $\Sigma_i$ lift to $(1+o_i(1))$-quasigeodesics in $\H^3$ we  obtain 
\begin{equation}\label{hausdorff.close.lemma}
d_H(D_i,\Omega_i)\to 0\quad\mbox{as }i\to\infty.
\end{equation}
Let $\mu_i$, $\nu_i$ denote  the unit Radon measure of $M$ induced by integration over $S_i$ and $\Sigma_i$, respectively.

\begin{lemm}\label{measures.same} $\lim_{i\to\infty}\mu_i=\lim_{i\to\infty}\nu_i$.
\end{lemm}
\begin{proof} Without loss of generality we can assume that both measures converge. Fix  $0<\eta<1$.

We saw in Proposition \ref{compactness-thm} that $D_i$ is locally area-minimizing mod 2 in $\H^3$ and so we have from \eqref{hausdorff.close.lemma} that for all $i$ sufficiently large and  every geodesic ball $B\subset \H^3$ of small radius
$$\area(D_i\cap B)\leq (1+\eta)\area(\Omega_i\cap B)$$
and thus from the fact that $\area(S_i)\area(\Sigma_i)^{-1}\to 1$ as $i\to\infty$ we have  $$\lim_{i\to\infty}\mu_i\leq (1+\eta) \lim_{i\to\infty}\nu_i.$$

Denote the set  $K_{\eta, i}\subset \Sigma_i$ simply by $K_i$ and let $\hat K_i \subset \Omega_i$ denote its pre-image. We have from \eqref{eta.constant.lemma} that for all $i$ sufficiently large, $\nu_i(M\setminus K_i)\leq 2\eta$. From the definition of  $\bar r$, the fact that geodesics in  $\Omega_i$ are $(1+o_i(1))$-quasigeodesics, and  \eqref{hausdorff.close.lemma} we have that for all $i$ sufficiently large and  all $x\in \hat K_i$,  $B_{\bar r}(x)\cap \Omega_i$ is very close to a geodesic disc of radius $\bar r$ in $D_i$. Thus for every geodesic ball $B\subset M$ of radius sufficiently small we obtain
$$\lim_{i\to\infty}\mu_i(B)\geq \lim_{i\to\infty}\nu_i(B\cap K_i)\geq \lim_{i\to\infty}\nu_i(B)-2\eta.$$
Making $\eta\to 0$ we deduce the result.
\end{proof}

In order to  prove property (b) of theorem we need to use some more properties related with the construction of $\Sigma_{\varepsilon}$.  Namely  the  nearly equidistribution  in the frame bundle $\mathcal F(M)$ of $M$ of the skew-pants that define $\Sigma_{\varepsilon}$.

An oriented $(R,\delta)$-skew pants is  defined as being the homotopy class of  some oriented $(R,\delta)$-geometric skew-pants immersion $f:P\rightarrow M$, where the homotopies  preserve the image and orientation of the boundary geodesics. The space $\mathcal P(R,\delta)$ of all such homotopy classes contains only finitely many elements.

Given a point $x=(p,(e_1,e_2,e_3))\in\mathcal F(M)$ we get a natural orientation in the $2$-plane  $V=span\{e_1,e_2\}\subset T_pM$ and an oriented ideal triangle $T(x)\subset V$  whose vertices  are
the endpoints of the geodesic ray based at $p$ with initial velocity $e_1$ and its $2\pi/3$ consecutive rotations in $U$ (see \cite[Section 4]{hamenstadt} for definitions: in the codimension one setting framed tripods and frames can be identified). For all $R$ large enough and $\delta$ small  it is shown in Lemma 7.4 of  \cite{hamenstadt}  (combined with Lemma 4.3  \cite{hamenstadt}) that for every pair $(x,y)\in \mathcal F(M)^2$ there are many  $(R,\delta)$-geometric skew-pants (which we denote by $P(x,y)$) with the property that, except from $3$ twisted bands whose area can be made very small, the rest of $P(x,y)$ is  arbitrarily close in Hausdorff distance  to the two ideal triangles $T(x), T(y)$.  One sees from the construction that these geometric skew-pants  have all the same geodesic boundary and are homotopic (this is explained in \cite[page 832]{hamenstadt}). Thus  we get a map  $\hat P:\mathcal F(M)^2\rightarrow P(R,\delta)$, where $\hat P(x,y)$ denotes the homotopy class of any of the $(R,\delta)$-geometric skew-pants $P(x,y)$.

Let $\lambda^2$ denote the normalized Lebesgue measure in $\mathcal F(M)^2$. For each $(R,\delta)$ consider the measure $\mu$ in  $ \mathcal F(M)^2$ that is obtained by integrating $d\mu$ defined in \cite[page 849]{hamenstadt} along the fiber $\mathcal F(M)^3$. It is absolutely continuous with respect to $\lambda^2$ and its Radon-Nikodym derivative has order $1+O(1/R)$. For each $P\in P(R,\delta)$ set $h(P)=\mu(\hat P^{-1}(P))$. In Proposition 7.3 of \cite{hamenstadt} we see that the surface $\Sigma_{\varepsilon}$ is constructed by attaching elements of $P(R,\delta)$ and if $n_P$ denotes the number of times that $P\in P(R,\delta)$ appears  in $\Sigma_{\varepsilon}$ then
\begin{equation}\label{h.lower.bound.claim}
n_P\geq \frac{h(P)}{2}\sum_{Q\in  P(R,\delta)}n_Q=\frac{h(P)}{2}2(g-1),
\end{equation}
where $g$ is the genus of $\Sigma_{\varepsilon}$.

The claim below and Lemma \ref{measures.same} proves Theorem \ref{counting.thm} (b).
\medskip

\noindent{\bf Claim:} {\em For every geodesic ball $B\subset M$ we have 
$\liminf_{i\to\infty}\nu_i(B)>0$.\\
}
It suffices to consider the case where each $\Sigma_i$ is one of the surfaces constructed in \cite{hamenstadt} (the finite covering case follows immediately). Choose  $\tilde B\subset B$ a smaller concentric ball in $M$ so that if $U$ denotes its preimage in $\mathcal F(M)$ we can find a constant $c_0$ (depending on $B$ and $\tilde B$) so that $\area(T(x)\cap B)\geq 2c_0$ for all $x\in U$. Thus for all $R$ large and $\delta$ small we have that $$\area(P(x,y)\cap B)\geq c_0$$ for all $x\in U$, $y\in \mathcal F(M)$. Set 
$$\Lambda=\{P\in P(R,\delta): \hat P^{-1}(P)\cap (U\times \mathcal F(M))\neq \emptyset \}.$$
Each time $P\in \Lambda$ choose its geometric representative to be $P(x,y)$ where $x\in U$. Therefore for all $i$ sufficiently large we have using \eqref{h.lower.bound.claim} that
\begin{align*}
\area(\Sigma_i\cap B)&\geq \sum_{P\in \Lambda}n_P\,\area(P\cap B)\geq c_0(g-1)\sum_{P\in \Lambda}h(P)\\
&\geq c_0(g-1)\mu(U\times \mathcal F(M))\geq c_0\mu(U\times \mathcal F(M))\area(\Sigma_i)/5\pi.\\ 
\end{align*}
This proves the claim.

\end{proof}

\section{Asymptotic inequality}\label{asymp.section}

Consider $\{S_i\}_{i\in\N}$ a sequence of minimal essential immersions given by Theorem \ref{counting.thm}, each  inducing a surface group $G_i<\Gamma$. For each $i\in\N$ consider as well the minimal essential immersion $\Sigma_i\subset M$ that minimizes area with respect to the metric $h$ in the homotopy class of $S_i$ (using \cite{schoen-yau} for instance).

The goal of this section (and the next) is to prove the following result.
\begin{thm}\label{rigidity.thm} Assume the metric $h$ has sectional curvature $\leq -1$. Then 
$$\limsup_{i\to\infty}\frac{\text{area}_h(\Sigma_i)}{\text{area}(S_i)}\leq 1.$$
If equality holds then the metric $h$ is hyperbolic.
\end{thm} 
\begin{proof}
Let $g_i$ denote the genus of $S_i$.  From Gauss equation we have that

\begin{equation}\label{gauss.eqn}
\area_{h}(\Sigma_i)=4\pi(g_i-1)+\int_{\Sigma_i}(K_{12}+1) dA_h-\frac{1}{2}\int_{\Sigma_i}|A|^2dA_h,
\end{equation}
where $K_{12}(x)$ is the ambient sectional curvature of  $T_x\Sigma_i$. Using the fact that $K_{12}\leq -1$ and \eqref{seppi.A} we have that
$$\limsup_{i\to\infty}\frac{\text{area}_h(\Sigma_i)}{\text{area}(S_i)}\leq 1$$
with equality implying that
\begin{equation*}\label{gauss.eqn.limit}
\lim_{i\to\infty}\frac{1}{\text{area}_h(\Sigma_i)}\int_{\Sigma_i} |A|^2-(K_{12}+1) dA_h=0.
\end{equation*}
Consider the nonnegative smooth function  $f_i=|A|^2-(K_{12}+1)$ on $\Sigma_i$. We then have
$$
\lim_{i\to\infty}\frac{1}{\text{area}_h(\Sigma_i)}\int_{\Sigma_i}|f_i|dA_h=0.
$$

In Section \ref{section.technical} we show (see Corollary \ref{cor.technical.main})  the existence of  a group $H_i<\Gamma$ conjugate to $G_i$ so that if $D_i$, $\Omega_i$ denote, respectively, the lifts of $S_i$ and $\Sigma_i$ to $\H^3$ that are preserved by $H_i$ we have, after passing to a subsequence, 
\begin{itemize}
\item[(i)] $\Lambda(H_i)$ converges  in Hausdorff distance, as $i\to\infty$, to $\gamma\in \mathcal C$  with $\Gamma\gamma$ dense in $\mathcal C$ and
\item[(ii)]  for all $R>0$
\end{itemize}
\begin{equation}\label{limit.zero}
\lim_{i\to\infty}\int_{\Omega_i\cap B_R(p)}|f_i|dA_h=0.
\end{equation}
From (i) we have that all $D_i$'s must intersect a compact set in $\H^3$ and so  Proposition \ref{compactness-thm} implies that $\{D_i\}_{i\in\N}$ converges to a totally geodesic disc $D$ for the hyperbolic metric with $\partial_{\infty}D=\gamma.$

Because the ambient curvature is negative, $\Sigma_i$ is negatively curved and so, in virtue of being essential, its injectivity radius has a uniform lower bound for all $i\in\N$. Hence, standard stability estimates imply that  the second fundamental form of $\Sigma_i$  is uniformly bounded for all $i\in \N$ along with all its derivatives. As a result we have from \eqref{limit.zero} that
\begin{equation}\label{limit.zero.pointwise}
\lim_{i\to\infty}\sup\{|A|(x)+|K_{12}(x)+1|:x\in \Omega_i\cap B_R(p)\}=0,\quad\text{all }R>0.
\end{equation}
 \begin{prop}\label{plateau} There is a totally geodesic disc $\Omega$ in $(B^3,h)$ with $\partial_{\infty}\Omega=\gamma$ and such that the sectional curvature of $T_x\Omega$ is $-1$ for all $x\in \Omega$.
 \end{prop}
 \begin{proof}
From Theorem \ref{convex.hull.main.thm} we obtain the existence of a compact set $K$ that intersects  $\Omega_i$ for  all $i\in\N$. Choose $x_i\in\Omega_i\cap K$ and denote by $B_R^i(x_i)\subset \Omega_i$  the intrinsic ball of radius $R$ centered at $x_i$. Note that $\Omega_i$ is negatively curved and thus $B_R^i(x_i)$ is diffeomorphic to a disc for all  $i$ sufficiently large. Standard compactness of minimal surfaces with uniform bounds on the second fundamental form gives the existence of a complete minimal surface $\Omega\subset B^3$ so that, after passing to a subsequence, intrinsic discs  in $\Omega_i$ centered at $x_i$ converge strongly to intrinsic discs in $\Omega$. Furthermore, from \eqref{limit.zero.pointwise}, we have that $\Omega$ is totally geodesic and the sectional curvature of $T_x\Omega$ is $-1$ for all $x\in \Omega$. As a result $\Omega$ is diffeomorphic to a disc. We have from Proposition \ref{convex.core} that $\Omega_i\subset C_h(\Lambda(H_i))$  for all $i$ sufficiently large and so $\partial_{\infty}\Omega\subset \gamma$. On the other hand, $\partial_{\infty}\Omega$ is homeomorphic to a circle and so it must be equal to $\gamma$.
 \end{proof}

Consider the following circle bundles
$$S_1^D:=\{(p,v):p\in D, v\in T_pD, \bar h(v,v)=1\}$$
and
$$S_1^\Omega:=\{(p,v):p\in \Omega, v\in T_p\Omega, h(v,v)=1\}.$$
Denote by $S_1M(\bar h)$ and $S_1M(h)$ the unit tangent bundle of $M$ with respect to $\bar h$ and $h$ respectively, and  let $S_1^D(M)\subset S_1M(\bar h)$, $S_1^\Omega(M)\subset S_1M(h)$ denote, respectively, the  projection to $S_1M(\bar h)$ and $S_1M(h)$  of $S_1^D$ and $S_1^\Omega$. From (i) we have that $S_1^D(M)$ is dense in $S_1M(\bar h)$.

We  now argue that the sectional curvature of every $2$-plane  in $(M,h)$  is $-1$.

\noindent{\bf Claim: }{\em For every $(p,v)\in S_1(M)(h)$ there is a totally geodesic hyperbolic disc $\Omega_{(p,v)}$ in $(B^3,h)$ whose projection in $M$ contains the geodesic passing through $p$ with direction $v$.
}
\medskip

From the  geodesic rigidity proven in Gromov \cite{gromov} there is a homeomorphism  $T$ from  $S_1M(\bar h)$ to $S_1M(h)$  that maps geodesics onto geodesics, meaning that if $\gamma$ is a geodesic in $(M,\bar h)$, there is a geodesic $\sigma$ in $(M,h)$ so that for all $t\in \R$ there is $s\in \R$ so that $T(\gamma(t),\gamma'(t))=(\sigma(s),\sigma'(s))$. Moreover,  from its proof (see \cite[Theorem 2.12]{knieper} for instance), $T$ can be chosen so that if $\gamma(+\infty), \gamma(-\infty)\in S^2_{\infty}$ are the asymptotes of $\gamma$, then $\sigma$ has the same asymptotes in  $S^2_{\infty}$. Thus, from the fact that $\partial_{\infty}\Omega=\partial_{\infty}D$ and that both $D$ and $\Omega$ are totally geodesic, we have that $T$ is also a homeomorphism from $S_1^D(M)$  onto $S_1^\Omega(M)$. Therefore, because $S_1^D(M)$ is dense in $S_1M(\bar h)$ we obtain that $S_1^\Omega(M)$ is also dense in $S_1M(h)$. As a result, for every $(p,v)\in S_1M(h)$ we can find  a sequence of points $\{\omega_i\}_{i\in\N}$ in $S_1^\Omega$ whose projection to  $S_1M(h)$ converges to $(p,v)$ and so applying the same reasoning as in Proposition \ref{compactness-thm} to a suitable sequence $\{\phi_i(\Omega)\}_{i\in\N}$, where $\phi_i\in\Gamma$,  we obtain  a totally geodesic hyperbolic disc $\Omega_{(p,v)}\subset B^3$  whose projection in $M$ contains the geodesic passing through $p$ with direction $v$.
\medskip

Recalling the discussion in Section \ref{subsection.geodesic.planes}, choose $(p,(e_1,e_2,e_3))\in \mathcal F(M)(h)$ whose orbit  under the frame flow
$$F_t((p,(e_1,e_2,e_3)))=(\gamma(t),(\gamma'(t),e_2(t),e_3(t))), \quad t\geq 0$$ is dense in $\mathcal F(M)(h)$. We abuse notation and  denote the lift of $\gamma$ to $B^3$ by $\gamma$. By applying a rotation if necessary, we can prescribe the vector $e_2$ to be any unit vector orthogonal to $e_1$ that we still obtain a dense orbit  in $\mathcal F(M)(h)$.   Hence we assume that $\{e_1,e_2\}$ span $T_{\gamma(0)}\Omega_{(p,e_1)}$, in which case the fact that $\Omega_{(p,e_1)}$ is totally geodesic implies that $\text{span}\{\gamma'(t),e_2(t)\}=T_{\gamma(t)}\Omega_{(p,e_1)}$  for all $t\geq 0$. Therefore the set of $2$-planes with sectional curvature $-1$ is dense in $Gr_2(M)$ and this implies the desired result.

\end{proof}
\section{Nearly totally geodesic minimal surfaces}\label{section.technical}

We continue assuming the set up of the last section. Namely we have   a sequence of minimal essential immersions $\{S_i\}_{i\in\N}$ given by Theorem \ref{counting.thm}, each  inducing a surface group $G_i<\Gamma$ and lifting to a disc $D_i\subset \H^3$ that is preserved by $G_i$.

 For each $i\in\N$ consider as well the minimal essential immersion $\Sigma_i\subset M$ that minimizes area with respect to the metric $h$ in the homotopy class of $S_i$ and such that  there is a continuous function $f_i:\Sigma_i\rightarrow \R$ so that
\begin{equation}\label{H2}
\lim_{i\to\infty}\frac{1}{\text{area}_h(\Sigma_i)}\int_{\Sigma_i}|f_i|dA_h=0.
\end{equation}
Let $\Omega_i$ denote the disc lifting $\Sigma_i$ to $B^3$ that is preserved by $G_i<\Gamma$, $i\in\N$. To make notation easier, it is understood that the function $f_i$ on $\phi(\Omega_i)$, $\phi\in\Gamma$, means $f_i\circ\pi_{\Omega_i}\circ\phi^{-1}$, where $\pi_{\Omega_i}$ is  the projection from $\Omega_i$ to $\Sigma_i$.

Fix $p\in \H^3$, consider  for every $\varepsilon, R>0$
$$F_i(\varepsilon, R)=\left\{\phi\in\Gamma: \int_{\phi(\Omega_i)\cap B_R(p)}|f_i|dA_h\leq \varepsilon\right\}$$
and define $\mathcal L\subset\mathcal C$  as
\begin{multline*}
\mathcal L=\{\gamma\in\mathcal C: \exists\,\phi_i\in F_i(\varepsilon_i, R_i)\text{ with }\varepsilon_i\to 0, R_i\to\infty \text{ so that,} 
\\
\text{after passing to a subsequence, }\Lambda(\phi_iG_i\phi_i^{-1})\text{ converges to }\gamma\}.
\end{multline*}
The goal of this section is to show
\begin{thm}\label{technical.main.thm} $\mathcal L=\mathcal C$ and so there is $\gamma\in \mathcal L$ so that $\Gamma\gamma$ is dense in $\mathcal C$.
\end{thm}
This result has the following corollary

\begin{cor}\label{cor.technical.main}
There is a conjugate group $H_i=\phi_iG_i\phi_i^{-1}$, $\phi_i\in\Gamma$, so that, after passing to a subsequence, 
\begin{itemize}
\item[a)] $\Lambda(H_i)$ converges  in Hausdorff distance, as $i\to\infty$, to $\gamma\in \mathcal C$  with $\Gamma\gamma$  dense in $\mathcal C$ and
\item[b)]  for all $R>0$
$$\lim_{i\to\infty}\int_{\phi_i(\Omega_i)\cap B_R(p)}|f_i|dA_h=0.$$
\end{itemize}
\end{cor}

\begin{proof}[Proof of Theorem \ref{technical.main.thm}]
We start by showing
\begin{lemm}The set $\mathcal L$ is closed and $\Gamma$-invariant.
\end{lemm}
\begin{proof}
The fact that it is closed follows by extracting a diagonal subsequence.

With $\psi\in\Gamma$, set $\alpha=d(p,\psi(p))$. Using the fact that $\psi^{-1}(B_{R-\alpha}(p))\subset B_R(p)$ the reader can check that, for all $R>0$ and all $\varepsilon>0$, 
\begin{equation}\label{inclusion.lemm}
 \phi\in F_i(\varepsilon,R)\implies \psi\phi\in F_i(\varepsilon,R-\alpha).
\end{equation}
Combining this with the fact that $\psi(\Lambda(H))=\Lambda(\psi H\psi^{-1})$ for every discrete subgroup $H\subset \Gamma$, follows at once that if $\gamma\in\mathcal C$ then $\psi(\gamma)\in\mathcal C$.
\end{proof}
Hence, it suffices to find $\gamma\in \mathcal L$ so that $\Gamma\gamma$ is dense in $\mathcal C$. Before we provide the details we describe first the general idea. The key step is to show that for every compact set $K\subset\H^3$ there is $\gamma\in\mathcal L$ so that $C(\gamma)$ intersects $K$. Indeed, if no dense orbit exists then every point in $\mathcal L$ is isolated (Theorem \ref{mmo-lemma}) and so we  can find a compact set $K$ so that $C(\gamma)$ never intersects $K$ for all $\gamma\in\mathcal L$, which is a contradiction.

Consider a Dirichlet fundament domain $p\in \Delta$ for $M$ so that $\partial\Delta$ is transverse to both $\phi(D_i)$ and $\phi(\Omega_i)$ for all $\phi\in\Gamma$.   We now consider $\Gamma^{S_i},$ $ \Gamma^{S_i}(K)$ to be the set of all lifts of $S_i$ that intersect $\Delta$, $K$, respectively,  $\Gamma^{\Sigma_i}$ to be the set of all lifts of $\Sigma_i$ that intersect $\Delta$, and $\Gamma^{\Sigma_i}(\varepsilon,R)$  to be the lifts in $\Gamma^{\Sigma_i}$ for which the function $|f_i|$ is small in $L^1$ on a ball of radius $R$. More precisely,
\begin{align*}&\Gamma^{S_i} =\{\phi\in\Gamma: \phi(D_i)\cap \Delta\neq \emptyset\},&  \Gamma^{S_i}(K) =\{\phi\in\Gamma: \phi(D_i)\cap K\neq \emptyset\},\\ 
&\Gamma^{\Sigma_i} =\{\phi\in\Gamma: \phi(\Omega_i)\cap \Delta\neq \emptyset\}, & \Gamma^{\Sigma_i}(\varepsilon,R) =F_i(\varepsilon,R)\cap\Gamma^{\Sigma_i}.
\end{align*}

We want to find $\varepsilon_i\to 0$, $R_i\to\infty$, so that $\Gamma^{S_i}(K)\cap F(\varepsilon_i,R_i)$ is always nonempty.

 The strategy is the following: The  sets described above are all invariant by right multiplication with $G_i$ because $G_i$ preserves both $D_i$ and $\Omega_i$.  We denote the projection of these sets in $\Gamma\setminus G_i$ by $\underline\Gamma^{S_i}$, $\underline\Gamma^{\Sigma_i}$,$\underline \Gamma^{S_i}(K)$, and $\underline \Gamma^{\Sigma_i}(\varepsilon,R)$. 
We will see that, for all $i$ very large,  $\#\underline\Gamma^{S_i}$ is proportional to $ \text{area}(S_i)$, use the fact that $d_H(\Omega_i,D_i)$ is bounded  to conclude that $\Gamma^{S_i}$ and $\Gamma^{\Sigma_i}$ are at a finite Hausdorff distance from each other,  deduce from Theorem \ref{counting.thm} b) that $\frac{\#\underline\Gamma^{S_i}(K)}{\#\underline\Gamma^{S_i}}$ is bounded below away from zero, and use \eqref{H2} to deduce that  $\frac{\#\underline \Gamma^{\Sigma_i}(\varepsilon,R)}{\#\underline\Gamma^{\Sigma_i}}\simeq 1$. Putting all these facts together one can then conclude that $\Gamma^{S_i}(K)\cap F(\varepsilon,R)\neq \emptyset$ for all $i$ very large. We now provide the details.

Referring to the notation set in  Section \ref{fundamental.domain},  we fix a representative $\underline \phi$ for each coset $\phi G_i\in \Gamma\setminus G_i$. Recall that $\nu$ is the measure given by  Theorem \ref{counting.thm} b)
\begin{prop}\label{properties.main}
There are constants $n=n(M,h)\in\N$, $\alpha=\alpha(M)>0$, and $\beta=\beta(\nu,K)>0$  so that for all $i$ sufficiently large
\begin{itemize}
\item[a)] $\mathcal{d}_H(\Gamma^{S_i},\Gamma^{\Sigma_i})\leq n$;
\item[b)] $\alpha^{-1}\text{area}(S_i)\leq \#\underline\Gamma^{S_i}\leq \alpha \text{area}(S_i)$;
\item[c)] $\liminf_{i\to\infty}\frac{\#\underline\Gamma^{S_i}(K)}{\#\underline\Gamma^{S_i}}\geq \beta.$
\end{itemize}
\end{prop}
\begin{proof}
From Theorem \ref{convex.hull.main.thm} we have the existence of $c_1=c_1(h)$ so that $d_H(\phi(D_i),\phi(\Omega_i))\leq c_1$ for all $\phi\in \Gamma$ for all $i$ sufficiently large and from Lemma \ref{cayley.balls} we have the existence of $n=n(M, c_1)\in \N$ so that 
$B_{c_1}(x)\subset \cup_{|\phi|\leq n}\phi(\Delta)$ for all $x\in\Delta.$

Choose $\psi\in\Gamma^{S_i}$ and pick $x\in \psi(D_i)\cap \Delta$. There is $y\in \psi(\Omega_i)\cap B_{c_1}(x)$ and thus some $\phi\in\Gamma$ with $|\phi|\leq n$ for which $\phi^{-1}(\psi(\Omega_i))\cap \Delta\neq\emptyset$. Hence $\Gamma^{S_i}$ is in a $n$-neighborhood of $\Gamma^{\Sigma_i}$  (for the distance $\mathcal d$) and reversing the roles of $\Sigma_i$ and $S_i$  proves a).

Recall from Section  \ref{fundamental.domain} that for all $\psi\in\Gamma$, $\Delta_i=\cup_{\underline\phi\in\Gamma\setminus G_i}\underline \phi^{-1}(\psi(\Delta))$ is a fundamental domain for $\H^3\setminus G_i$. Thus
\begin{equation}\label{fund.expression}
\text{area}\,(S_i)=\sum_{\underline \phi\in\Gamma\setminus G_i}\text{area}\,(\underline \phi(D_i)\cap \psi(\Delta)).
\end{equation}
Choose $A\subset \Gamma$ a finite set so that a neighborhood of radius $1$ of $\Delta$ is contained in the interior of $\cup_{\psi\in A}\psi(\Delta)$.  If $x\in \underline\phi(D_i)\cap\Delta$ then we have from the monotonicity formula that, for some $c_2=c_2(M)$,
$$c_2\leq \text{area}\,(\underline\phi(D_i)\cap B_1(x))\leq\sum_{\psi\in A}\text{area}\,(\underline\phi(D_i)\cap \psi(\Delta))$$
and so, using \eqref{fund.expression},
\begin{multline*}
c_2 \# \underline\Gamma^{S_i}\leq \sum_{\psi\in A}\sum_{\underline \phi\in\underline\Gamma^{S_i}}\text{area}\,(\underline\phi(D_i)\cap \psi(\Delta))\\
\leq \sum_{\psi\in A}\sum_{\underline \phi\in \Gamma\setminus G_i}\text{area}\,(\underline\phi(D_i)\cap \psi(\Delta))=\# A \,\text{area}\,(S_i)
\end{multline*}
Applying Proposition \ref{compactness-thm} to any sequence $\underline\phi_i(D_i)$ with $\underline \phi_i\in\underline\Gamma^{S_i}$ we obtain the existence of a constant $c_3=c_3(M)$ so that for all $i$ sufficiently large we have
\begin{equation}\label{eq.area.estimate}
\text{area}\,(\underline \phi(D_i)\cap\Delta)\leq c_3\quad\text{for all }\underline\phi\in \underline\Gamma^{S_i}.
\end{equation}
Thus
$$\text{area}\,(S_i)=\sum_{\underline\phi\in \Gamma\setminus G_i}\text{area}\,(\underline \phi(D_i)\cap\Delta)
=\sum_{\underline\phi\in \underline\Gamma^{S_i}}\text{area}\,(\underline \phi(D_i)\cap\Delta)\leq c_3\#\,\underline\Gamma^{S_i}$$
and hence for all $i$ sufficiently large
$$ \frac{1}{c_3} \text{area}\,S_i\leq \# \underline\Gamma^{S_i}\leq \frac{\# A}{c_2}\text{area}\,S_i.$$
This proves b).

Let $f\in C^0(M)$ be a  function with $0\leq f\leq 1$ and support contained in $K$. Using \eqref{eq.area.estimate} we have that for all $i$ sufficiently large
$$
\int_{S_i}fdA=\sum_{\underline \phi\in \Gamma\setminus G_i}\int_{\underline \phi(D_i)\cap \Delta}fdA=\sum_{\underline \phi\in \underline\Gamma^{S_i}(K)}\int_{\underline \phi(D_i)\cap \Delta}fdA\leq  c_3\#\underline \Gamma^{S_i}(K)
$$
which means that
$$c_3\frac{\# \underline\Gamma^{S_i}(K)}{\# \underline\Gamma^{S_i}}\geq \frac{1}{\text{area}\,(S_i)}\int_{S_i}fdA$$
and this proves c).

\end{proof}
In light of Proposition \ref{properties.main} a) we can construct, for all $i$ sufficiently large, a map  $P_i:\Gamma^{S_i}\rightarrow \Gamma^{\Sigma_i}$ so that
\begin{itemize}
\item[i)] $\mathcal{d}(P_i(\phi),\phi)\leq n$ for all $\phi\in \Gamma^{S_i}$;
\item[ii)] $P_i(\phi g)=P_i(\phi) g$ for all $\phi\in \Gamma^{S_i}, g\in G_i$.
\end{itemize}
Set $\Gamma^{S_i}(\varepsilon, R)=P_i^{-1}(\Gamma^{\Sigma_i}(\varepsilon, R))$. Because the map $P_i$ is $G_i$-invariant then $\Gamma^{S_i}(\varepsilon, R)$ is also $G_i$-invariant and $P_i$ descends to map $\underline P_i:\underline\Gamma^{S_i}\rightarrow \underline\Gamma^{\Sigma_i}$.
\begin{prop}\label{properties2.main}
 For all $\varepsilon>0, R>0,$
$$\liminf_{i\to\infty}\frac{\#\underline\Gamma^{S_i}(\varepsilon,R)}{\#\underline\Gamma^{S_i}}=1.$$
\end{prop}
\begin{proof}
Due to the fact that both $S_i$ and $\Sigma_i$ minimize area in their homotopy class, there is a constant $c_1=c_1(h)$ so that $$c_1^{-1}\text{area}_h(\Sigma_i)\leq \text{area}(S_i)\leq c_1\,\text{area}_h(\Sigma_i)$$ for all $i\in\N$ and so we deduce from Proposition \ref{properties.main} b) the existence of $c_2=c_2(h,M)$ so that, for all $i$ sufficiently large,
\begin{equation}\label{area.comparison}
c_2^{-1}\text{area}_h(\Sigma_i)\leq \#\underline\Gamma^{S_i}\leq c_2\,\text{area}_h(\Sigma_i).
\end{equation}

Set $L^{\Sigma_i}(\varepsilon, R):=\Gamma^{\Sigma_i}-\Gamma^{\Sigma_i}(\varepsilon,R)$, $i\in\N$ and denote its projection to $\Gamma\setminus G_i$ by $\underline L^{\Sigma_i}(\varepsilon, R)$. From Lemma \ref{cayley.balls} there is $n_R=n_R(R,M)$ so that $$B_R(p)\subset \cup_{|\psi|\leq n_R}\phi(\Delta)$$ and set $c_3=\#\{\psi\in\Gamma:|\psi|\leq n_R\}$. Then, recalling that $$\Delta_i=\cup_{\underline\phi\in\Gamma\setminus G_i}\underline \phi^{-1}(\psi(\Delta))$$ is a fundamental domain for $\H^3\setminus G_i$ for all 
$\psi\in\Gamma$, we have
\begin{multline*}
c_3\int_{\Sigma_i}|f_i|dA_h=\sum_{|\psi|\leq n_R}\sum_{\underline\phi\in \Gamma\setminus G_i}\int_{\underline \phi(\Omega_i)\cap \psi(\Delta)}|f_i|dA_h\\
= \sum_{\underline\phi\in \Gamma\setminus G_i}\sum_{|\psi|\leq n_R}\int_{\underline \phi(\Omega_i)\cap \psi(\Delta)}|f_i|dA_h \geq  \sum_{\underline\phi\in \Gamma\setminus G_i}\int_{\underline \phi(\Omega_i)\cap B_R(p)}|f_i|dA_h\\
\geq \sum_{\underline\phi\in\underline  L^{\Sigma_i}(\varepsilon, R)}\int_{\underline \phi(\Omega_i)\cap B_R(p)}|f_i|dA_h\geq \varepsilon\#\underline L^{\Sigma_i}(\varepsilon, R).
\end{multline*}
Hence
$$\frac{\#\underline L^{\Sigma_i}(\varepsilon, R)}{\text{area}_h(\Sigma_i)}\leq\frac{c_3}{\varepsilon\text{area}_h(\Sigma_i)}\int_{\Sigma_i}|f_i|dA _h$$
and we deduce from \eqref{H2} and \eqref{area.comparison} that
$$\liminf_{i\to\infty}\frac{\#\underline L^{\Sigma_i}(\varepsilon,R)}{\#\underline\Gamma^{S_i}}=0.$$
Set, for all $i$ sufficiently large, $L^{S_i}(\varepsilon,R)=P_i^{-1}(L^{\Sigma_i}(\varepsilon,R))$ which has its projection to $\Gamma\setminus G_i$ satisfying $\underline L^{S_i}(\varepsilon,R)=\underline P_i^{-1}(\underline L^{\Sigma_i}(\varepsilon,R))$.

Define $c_4=\#\{\phi\in \Gamma:|\phi|\leq n\}$, where $n$ is the constant in Proposition \ref{properties.main} a). From property i) of the map $P_i$ we have that $\#P_i^{-1}(\psi)\leq c_4$ for all 
$\psi\in \Gamma^{\Sigma_i}$. Hence from property ii) we deduce that $\# \underline L^{S_i}(\varepsilon,R)\leq c_4\# \underline L^{\Sigma_i}(\varepsilon,R)$ and we obtain
$$\liminf_{i\to\infty}\frac{\#\underline L^{S_i}(\varepsilon,R)}{\#\underline\Gamma^{S_i}}=0.$$
The desired result follows because the reader can check that $\underline \Gamma^{S_i}(\varepsilon,R)=\underline \Gamma^{S_i}-\underline L^{S_i}(\varepsilon,R)$.
\end{proof}

This proposition allows us to choose $\varepsilon_i\to 0$ and $R_i\to\infty$ as $i\to 0$ so that
\begin{equation}\label{full.measure}
\liminf_{i\to\infty}\frac{\#\underline\Gamma^{S_i}(\varepsilon_i,R_i)}{\#\underline\Gamma^{S_i}}=1.
\end{equation}

\begin{lemm}\label{technical.main.lemm} There is a constant $c=c(M,h)$ so that for every compact set $K$ contained in $\Delta$ we can find $\{\phi_i\}_{i\in\N}\subset \Gamma$ so that for all $i$ sufficiently large $\phi_i\in\Gamma^{S_i}(K)\cap F_i(\varepsilon_i,R_i-c)$.
\end{lemm}
\begin{proof}From Proposition \ref{properties.main} c) and \eqref{full.measure} we can choose $\{\phi_i\}_{i\in\N}\subset \Gamma$ so that for all $i$ sufficiently large $$\phi_i\in \Gamma^{S_i}(\varepsilon_i,R_i)\cap \Gamma^{S_i}(K).$$ 
Thus from the definition of $P_i$ there is $g_i\in\Gamma$ with $|g_i|\leq n$ so that $g_i\phi_i\in F_i(\varepsilon_i,R_i)$. 
Set $c=\max\{d(p,\phi(p)):|\phi|\leq n\}$. Then from \eqref{inclusion.lemm} we have that $\phi_i\in F_i(\varepsilon_i,R_i-c)$ for all $i$ sufficiently large.
\end{proof}

Suppose that $\mathcal L$ has no element with a dense $\Gamma$-orbit in $\mathcal C$. Then Theorem \ref{mmo-lemma} implies that every point in $\mathcal L$ is isolated and so the set $$\{\gamma\in\mathcal L:C(\gamma)\cap \Delta\neq \emptyset\}$$
is finite. Thus, because  every $\gamma\in \mathcal L$ has $C(\gamma)$ projecting to a closed surface in $M$, we can choose a compact set $K\subset \Delta$ so that $C(\gamma)\cap  K=\emptyset$ for all $\gamma\in \mathcal L$. On the other hand, applying Theorem \ref{compactness-thm} to $\phi_i(D_i)
$, where the sequence $\{\phi_i\}_{i\in\N}\subset \Gamma$ is the one given by Lemma \ref{technical.main.lemm}, we obtain $\gamma\in\mathcal L$ for which $C(\gamma)\cap K\neq \emptyset$, which is a contradiction.
 \end{proof}

\section{Proof of Theorem \ref{asymp.area}}
This section is devoted to the proof of Theorem \ref{asymp.area}. {Recall that given a closed Riemannian manifold $(N,g)$, the volume entropy is defined as
$$E_{vol}(g)=\lim_{R\to\infty}\frac{\ln vol(\widehat B_R(x))}{R}= \lim_{R\to\infty}\frac{\ln \#\{\gamma\in \pi_1(N):\hat d(x,\gamma(x))\leq R\}}{R},$$  
where $\widehat B_R(p)$ and $\hat d$ denote, respectively, the geodesic balls and distance function induced by $g$ in the universal cover $\hat N$ of $N$. The fact that the first limit exists was first observed in \cite{manning} and it is standard to check the second identity.

Suppose we have an essential immersion $\Sigma\subset M$ which lifts to a disc $\Omega$ in the universal cover $\hat M$ of $M$. In this case $\pi_1(\Sigma)$ acts naturally by isometries in $\hat M$ and  if $\hat d_{\Omega}$  denotes the intrinsic distance in $\Omega$ we have $\hat d(x,y)\leq \hat d_{\Omega}(x,y)$ for all $x,y\in \Omega$. Thus
\begin{align*}
\#\{\gamma\in \pi_1(\Sigma):\hat d_{\Omega}(x,\gamma(x))\leq R\} \leq & \#\{\gamma\in \pi_1(\Sigma):\hat d(x,\gamma(x))\leq R\}\\
 \leq &\#\{\gamma\in \pi_1(M):\hat d(x,\gamma(x))\leq R\}.
\end{align*}
Hence $E_{vol}(h_{\Sigma})\leq E_{vol}(h)$.  From \cite{BCG} we have $E_{vol}(h_{\Sigma})^2\text{area}_h(\Sigma)\geq 4\pi(g-1)$, where $g$ is the genus of $\Sigma$ and so by minimizing area in the homotopy class $\Pi$ of $\Sigma$ we deduce that
$$\text{area}_{h}(\Pi)\geq  E_{vol}(h)^{-2}4\pi(g-1).$$
Thus,  denoting by $\floor{x}$ the integer part of $x$,
$$\text{area}_h(\Pi)\leq 4\pi(L-1)\implies \Pi\in S(M,\floor{E_{vol}(h)^{2}L}),$$
and so, for all $\varepsilon>0$ and all $L$ sufficiently large, we have from Theorem \ref{counting.thm} that
\begin{multline*}
\ln\#\{\text{area}_h(\Pi)\leq 4\pi(L-1):\Pi\in S_{\varepsilon}(M)\}\leq \ln s(M,\floor{E_{vol}(h)^{2}L},\varepsilon)\\
\leq 2E_{vol}(h)^{2}L\ln(c_2E_{vol}(h)^{2}L),
\end{multline*}
which implies that $E(h)\leq 2E_{vol}(h)^{2}$. Next we compute  $E(\bar h)$.
}

Given $\Pi\in S_{\varepsilon}(M)$, consider the essential minimal surface  $S\in \Pi$ so that $\text{area}(S)=\text{area}(\Pi)$. From Theorem \ref{compactness-thm}  we have $|A|^2_{L^{\infty}(S)}=o_{\varepsilon}(1)$, meaning that if $\varepsilon$ is very small then the quantity on the left side will also be small. Let $g$ be the genus of $S$. The integrated form of Gauss's equation \eqref{gauss.eqn} gives
\begin{equation*}
\text{area}(S)=4\pi(g-1)+o_{\varepsilon}(1)\text{area}(S)
\end{equation*}
and so for all $\varepsilon$ uniformly small we have
\begin{equation}\label{area.minimal}
\text{area}(S)=4\pi(g-1)(1+o_{\varepsilon}(1)).
\end{equation}
One immediate consequence is that, given $\delta>0$,  for all $\varepsilon$ small and all $L$ large  (depending on $\delta$ but independently of $\Pi$) we have both 
\begin{align*}
\text{area}(\Pi)\leq 4\pi(L-1)\text{ and } \Pi\in S_{\varepsilon}(M)& \implies \Pi\in S(M,\floor{(1+\delta)L},\varepsilon),\\
\Pi\in S(M,\floor{(1-\delta)L},\varepsilon)& \implies \text{area}(\Pi)\leq 4\pi(L-1)
\end{align*}
and so, recalling the notation set in Section \ref{black.box},
\begin{multline*}
\ln s(M,\floor{(1-\delta)L},\varepsilon)\leq \ln\#\{\text{area}(\Pi)\leq 4\pi(L-1):\Pi\in S_{\varepsilon}(M)\}\\
\leq \ln s(M,\floor{(1+\delta)L},\varepsilon).
\end{multline*}
Combining with Theorem \ref{counting.thm} we deduce that for all $\varepsilon$ small
$$2(1-\delta)\leq \liminf_{L\to\infty}\frac{\ln \#\{\text{area}(\Pi)\leq 4\pi(L-1):\Pi\in S_{\varepsilon}(M)\}}{L\ln L}\leq 2(1+\delta). $$
The arbitrariness of $\delta$ shows that $E(\bar h)=2$.

Suppose now that the sectional curvature of $h$ is less than or equal to $-1$. From the integrated form of Gauss's equation \eqref{gauss.eqn} we have that  every  genus $g$ minimal surface has $\text{area}_h(\Sigma)\leq 4\pi(g-1)$. Thus $\Pi\in S(M,\floor{L},\varepsilon)$ implies that $\text{area}_h(\Pi)\leq 4\pi(L-1)$. Hence
$$\#\{\text{area}_h(\Pi)\leq 4\pi(L-1):\Pi\in S_{\varepsilon}(M)\}\geq s(M,\floor{L},\varepsilon)$$
and  so Theorem \ref{counting.thm} implies that $E(h)\geq 2=E(\bar h)$.

Suppose now that $E(h)=2$. Consider the set $G(M,g,\varepsilon)\subset S(M,g,\varepsilon)$ given by Theorem \ref{counting.thm}. 
\medskip

\noindent{\bf Claim: }{\em For all $\delta>0$, there is $j\in\N$ so that for all $i\geq j$ we can find $g\in\N$ and  $\Pi\in G(M,g,1/i)$ so that
$$\text{area}_h(\Pi)>4\pi((1+\delta)^{-1}g-1).$$
}
Suppose not. In that case there is an increasing sequence of integers $\{i_j\}_{j\in\N}$ so that for all $g\in\N$ and $\Pi\in G(M,g,i_j^{-1})$ we have
$$ \text{area}_h(\Pi)\leq 4\pi((1+\delta)^{-1}g-1)$$
and hence, for all $L\geq 0$ 
$$\Pi\in G(M,\floor{(1+\delta)L},i_j^{-1})\implies  \text{area}_h(\Pi)\leq 4\pi(L-1).$$
Thus, for all $j\in\N$,
\begin{multline*}
\liminf_{L\to\infty}\frac{\ln \#\{\text{area}_h(\Pi)\leq 4\pi(L-1):\Pi\in S_{i_j^{-1}}(M)\}}{L\ln L}\\
 \geq \liminf_{L\to\infty}\frac{\ln \# G(M,\floor{(1+\delta)L},i_j^{-1})}{L\ln L}\geq 2(1+\delta),
\end{multline*}
which contradicts $E(h)=2$.
\medskip

Therefore we can find  an increasing  sequence of integers  $\{j_i\}_{i\in\N}$ and a sequence $\Pi_i\in G(M,g_i,j_i^{-1})$, $i\in \N$,  so that 
\begin{equation}\label{asym.limit}
\text{area}_h(\Pi_i)\geq 4\pi((1-1/i)g_i-1)\quad\text{for all }i\in\N.
\end{equation}
Denote by $S_i, \Sigma_i$ the minimal surfaces that minimize area in the homotopy class $\Pi_i$ with respect to $h$ and $\bar h$ respectively. We have $\text{area}(S_i)\leq 4\pi(g_i-1)$ and so we deduce from \eqref{asym.limit} that
\begin{equation}\label{asymp.identity}
\liminf_{i\to\infty}\frac{\text{area}_h(\Sigma_i)}{\text{area}(S_i)}\geq\liminf_{i\to\infty}\frac{4\pi((1-1/i)g_i-1)}{4\pi(g_i-1)}=1.
\end{equation}
The rigidity follows from  Theorem \ref{rigidity.thm}. \bibliographystyle{amsbook}

\end{document}